\documentclass[a4paper,oneside, reqno]{amsart}
\usepackage{mathrsfs}
\usepackage{amsfonts}
\usepackage{amssymb}
\usepackage{amsxtra}
\usepackage{dsfont}
\usepackage{amsthm}

\usepackage[english,polish]{babel}

\newcommand{\pl}[1]{\foreignlanguage{polish}{#1}}

\usepackage{color}

\usepackage{enumitem}
\usepackage{hyperref}
\usepackage{stmaryrd} 

\theoremstyle{plain}
\newtheorem{thm}[equation]{Theorem}
\newtheorem{cor}[equation]{Corollary}

\newtheorem{prop}[equation]{Proposition}

\newtheorem{lem}[equation]{Lemma}

\numberwithin{equation}{section}

\theoremstyle{remark}
\newtheorem{rem}[equation]{Remark}

\newcommand{\RR}{\mathbb{R}}
\newcommand{\ZZ}{\mathbb{Z}}

\newcommand{\CC}{\mathbb{C}}
\newcommand{\NN}{\mathbb{N}}

\newcommand{\X}{X}

\newcommand{\e}{\varepsilon}

\newcommand{\ind}[1]{{\mathds{1}_{{#1}}}}

\newcommand{\ex}[1]{\boldsymbol{e}\left( #1 \right)}

\def\R{{\mathbb{R}}}

\def\N{{\mathbb{N}}}
\def\Z{{\mathbb{Z}}}

\newcommand{\floor}[1]{{\lfloor {#1} \rfloor}}

\newcommand{\8}{\infty}

\title[Waring problem with a multiplicative error term: dimension-free estimates]
{On the solution of Waring problem with a multiplicative error term: dimension-free estimates}

\author{Jarosław Mirek}
\address[Jarosław Mirek]{Instytut Matematyczny\\
	Uniwersytet \pl{Wroc{\lll}awski}\\
	Plac Grun\-waldzki 2\\
	50-384 \pl{Wroc{\lll}aw}\\
	Poland}
\email{jd.mirek@gmail.com}

\author{Wojciech S{\l}omian}
\address[Wojciech S{\l}omian]{
	  Faculty of Pure and Applied Mathematics, 
	  Wroc\l{}aw University of Science and Technology\\
	  Wyb{.} Wyspia\'nskiego 27,
	  50-370 Wroc\l{}aw, Poland}
    \email{wojciech.slomian@pwr.edu.pl}

\author{B{\l}a{\.z}ej Wr{\'o}bel}
\address[B{\l}a{\.z}ej Wr{\'o}bel]{Instytut Matematyczny\\
	Uniwersytet \pl{Wroc{\lll}awski}\\
	Plac Grun\-waldzki 2\\
	50-384 \pl{Wroc{\lll}aw}\\
	Poland}
\email{blazej.wrobel@math.uni.wroc.pl}

\thanks{Jarosław Mirek, Wojciech S{\l}omian, and B{\l}a{\.z}ej Wr{\'o}bel were supported by the National
Science Centre, Poland, grant Opus 2018/31/B/ST1/00204. The authors thank the reviewers for their careful reading of the manuscript and helpful suggestions.}

\begin{document}
 
\selectlanguage{english}

\begin{abstract}
We give a relation between the radius and the dimension in which the asymptotic formula in the Waring problem holds in a multiplicative and dimension-free fashion.
\end{abstract}

\maketitle

\section{Introduction}
\label{sec:1}

Let $\N_0$ denote the set of nonnegative natural numbers $\N_0=\{0,1,2,\ldots,\}.$
In 1770 Waring made the statement that for each $k\in \N$ there exist $d\in \N$ such that every natural number $N$ can be written as
\begin{equation}\label{wareq}
    N=n_1^k+n_2^k+\ldots+n_d^k,\quad\text{for}\,\, n_i\in\NN_0.
\end{equation}
In the same year Lagrange gave a positive answer to the case $k=4$ of Waring's problem.
The first proof which concerns every $k\in\N$ was given by Hilbert \cite{Hil} in 1909. Hilbert's proof relied on, what is today called, Hilbert's identity which follows from some algebraic properties of polynomials. In the 1920' Hardy and Ramanujan \cite{HR} and Hardy and Littlewood \cite{HL1} began the study of questions related to Waring's problem from a more quantitative perspective. Namely, for $N\in \N$ by $r_k(N)$ we denote the number of $d$-tuples $(n_1, n_2, \ldots, n_d)\in\N^d$ which solve the equation \eqref{wareq}. In \cite{HR,HL1} the authors were interested in the asymptotic size of $r_k(N)$ as $N$ approaches infinity. 

The circle method was pioneered by Hardy and Littlewood in \cite{HL1} to prove  that for $k\ge2$,  $d\ge2^k+1$ 
we have
\begin{align}\label{asympto1}
	r_k(N)=\mathfrak{S}(N)\frac{\Gamma\left(1+\frac{1}{k}\right)^d}{\Gamma\left(\frac{d}{k}\right)}
	N^{d/k-1}+\mathcal{O}\big(N^{d/k-1-\delta}\big),
\end{align}
for some $\delta>0.$ Here $\Gamma$ is the Gamma function, $\mathfrak{S}(N)$ is the \textit{singular series} given by
\begin{equation}
	\label{eq:sise}
	\mathfrak{S}(N)=	\mathfrak{S}_d^k(N):=\sum_{q=1}^{\8}\sum_{\genfrac{}{}{0pt}{}{a=1}{(a, q)=1}}^qG(a/q)^d\ex{-Na/q}
\end{equation}
with $\ex{z}=\exp(2\pi {\boldsymbol{i}} z),$ and $G(a/q)$ is the Gaussian sum
\begin{align*}
	G_k(a/q)=G(a/q):=\frac{1}{q}\sum_{r=1}^{q}\ex{\frac{a}{q}r^k}.
\end{align*}
The asymptotic formula \eqref{asympto1} holds for 'large enough' $N,$ depending on $k$ and $d.$ 

The purpose of our note is to establish for a fixed $k\in \NN$ a range of $N$ (depending on the dimension $d$) with a precise meaning for 'large enough' for which the asymptotic formula holds.
\begin{thm}\label{thm:main}
For each $k\in \NN$ with $k\ge 2$ there exists a constant $C_k>0$ such that, for $d\ge 2^{k+1}$ and $N\ge C_k d^{20d}$, we have   
	\begin{equation}
		\label{eq:thm:main}
			r_k(N)=\mathfrak{S}(N)\frac{\Gamma\left(1+\frac{1}{k}\right)^d}{\Gamma\left(\frac{d}{k}\right)}N^{d/k-1}(1 + o(1)),
	\end{equation}
	where the implicit constants in the symbol $o(1)$ are independent of the dimension $d.$ 
\end{thm}

In other words Theorem \ref{thm:main} states that for $N\ge C_k d^{20d}$ one has a dimension-free asymptotic formula for $r_k(N)$ in the Waring problem \eqref{wareq} with a multiplicative error term. Apart from the case $k=2$ we have not found in the existing literature anything about uniformities with respect to radii and dimensions in \eqref{eq:thm:main}. The threshold $N\ge C_k d^{20d}$ is rather weak and surely far from being optimal in any sense. In particular, in the case $k=2$, Mirek, Szarek, and the third author \cite{miszwr1} have recently obtained a much better range of $N$ for a closely related Waring problem. Namely, let $\tilde{r}_2(N)$ denote the number of $d$-tuples of integer numbers $(n_1, n_2, \ldots, n_d)\in\Z^d$ which solve the equation
$n_1^2+n_2^2+\ldots+n_d^2=N.$ 
It was proved in \cite[Theorem 3.1, eq. 3.17]{miszwr1} that for $d\geq 16$ one has
\begin{equation*}
	\tilde{r}_2(N)=\mathfrak{S}(N)\frac{\pi^{d/2}}{\Gamma(d/2)}N^{d/2-1}\big(1+o(1)\big),
\end{equation*}
whenever $N> C d^3.$ The proof from \cite{miszwr1} can be repeated almost verbatim to give a similar formula in our case  
\begin{equation*}
r_2(N)=\mathfrak{S}(N)\frac{\pi^{d/2}}{2^d\Gamma(d/2)}N^{d/2-1}\big(1+o(1)\big)=\mathfrak{S}(N)\frac{\Gamma\left(3/2\right)^d}{\Gamma\left(d/2\right)}N^{d/2-1}(1 + o(1)).
\end{equation*}
The error term $o(1)$ in the above formula goes to zero as $N\to \infty$ in a dimension-free fashion in the range $N>C d^3.$
\begin{rem}
    When $k$-spheres are replaced by $k$-balls one can obtain a huge improvement over the threshold $N\ge C_k d^{20d}$. Indeed, let 
\begin{equation*}
    B_k(N):=\Big\{x\in \RR^d\colon \sum_{j=1}^d |x_j|^k\leq N\Big\}
\end{equation*}
be the closed unit ball of radius $N^{1/k}$ with respect to the $k$-norm and denote by $\tilde{b}_k(N)$ the number of  $d$-tuples of integer numbers $(n_1, n_2, \ldots, n_d)\in\Z^d$ which belong to $B_k(N).$
Then, from \cite[Lemmas 3.1 and 3.3]{KMPW} (general $k\ge 2$) and \cite[Lemmas 3.1 and 3.3]{BMSW4} ($k=2$) it follows that
\begin{equation}
\label{eq:ball}
c_k |B_k(N)|\leq\tilde{b}_k(N)\leq C_k |B_k(N)|,
\end{equation}
whenever $N\ge T_k d^k,$ with $T_k>0$ being a constant depending only on $k$. The symbol $|B_k(N)|$ in \eqref{eq:ball} denotes the $d$-dimensional Lebesgue measure of $B_k(N)$ while $c_k$ and $C_k$  are constants independent of the dimension $d.$ Note that \eqref{eq:thm:main} together with Proposition \ref{prop:singest} imply an estimate in the spirit of \eqref{eq:ball} only within the regime $N\ge C_k d^{20d}.$ A property that is crucial in establishing \eqref{eq:ball} is the nestedness of the $B_k(N)$ balls, which is obviously not available in the present context. 
\end{rem}

The methods employed in the proof of Theorem~\ref{thm:main} are the standard application of the Hardy--Littlewood circle method. Our contribution consists of carefully checking what is the smallest $N$ (in terms of $d$) for which the classical method works. The analysis is naturally split into the consideration of major arcs and minor arcs. In fact for a most part of the proof a much weaker condition $N> C_k d^{ck^2}$ is sufficient. The only part where we really need to restrict to $N>C_k d^{20d}$ lies in the major arc part. More precisely, the replacement of the exponential sum by an exponential integral is problematic in Lemma~\ref{lem:M1}, see Remark~\ref{rem:problem}. In \cite{miszwr1} this problem does not appear due to the use of different methods based on Poisson summation formulas for theta functions. It seems crucial there that $k=2.$
\section{Notation}
We now set up notation and terminology that will be used throughout the paper. The parameter $k\in\NN$ is fixed throughout the paper and will always mean the power in Waring's problem \eqref{wareq}. Therefore whenever we say that a constant is universal we mean that it may only depend on $k.$
\subsection{Basic notation} The sets $\ZZ, \RR$ and $\CC$ have the standard meaning.
\begin{enumerate}[label*={\arabic*}.]
    \item We denote $\NN:=\{1,2,\ldots\}$ and by $\mathbb{P}$ we mean the set of all prime numbers.
    \item For two natural numbers $a,q$  the greatest common divisor of $a$ and $q$ is denoted by $(a,q)$.
    \item For any $x\in\RR$ we will use the floor function
    \[
        \floor{x}: = \max\{ n \in \ZZ : n \le x \}.
    \]
    Moreover, for $N\in \NN$ we denote
    \[
    \X_N:=\floor{N^{1/k}}.
    \]
    \item We will use the convention that $\ex{z}:=e^{2\pi {\boldsymbol{i}} z}$ for every $z\in\CC$, where $\boldsymbol{i}^2=-1$. For $X\in \NN$ and $\xi \in \RR$ we denote 
    \[
        f_{X}(\xi):=\sum_{n=1}^{X}\ex{\xi n^k}.
    \]
    \item  For any $x>0$ we will use the Gamma function given by
    \[
        \Gamma(x):=\int_0^\infty{t^{x-1}e^{-t}{\rm d}t}.
    \]
    \item For $k\ge 2,$ $k\in \NN$ we also set
    \[
    \sigma(k):=\frac{1}{k(k-1)}.
    \]
\end{enumerate}
\subsection{Asymptotic notation and magnitudes}
\begin{enumerate}[label*={\arabic*}.]
\item For two nonnegative quantities
$A, B$ we write $A \lesssim_{\delta} B$ ($A \gtrsim_{\delta} B$) if
there is an absolute constant $C_{\delta}>0$ (which possibly depends
on $\delta>0$) such that $A\le C_{\delta}B$ ($A\ge C_{\delta}B$).  We
will write $A \simeq_{\delta} B$ when $A \lesssim_{\delta} B$ and
$A\gtrsim_{\delta} B$ hold simultaneously. We will omit the subscript
$\delta$ if it is irrelevant.

\item  For two nonnegative quantities $A, B$ we will also use the convention that $A \lesssim^{d} B$ ($A \gtrsim^{d} B$) if
there is an absolute constant $C>0$ such that $A\leq C^dB$ ($A\geq C^d B$). We
will write $A \simeq^{d} B$ when $A \lesssim^{d} B$ and
$A\gtrsim^{d} B$ hold simultaneously. Sometimes we abuse notation and write $A \lesssim_{\delta}^d B$ ($A \gtrsim_{\delta}^d B$) if
there is an absolute constant $C_{\delta}>0$ (which possibly depends
on $\delta>0$) such that $A\le C_{\delta}^dB$ ($A\ge C_{\delta}^dB$). In the same manner we will write $A \simeq_{\delta}^d B$ when $A \lesssim_{\delta}^d B$ and
$A\gtrsim_{\delta}^d B$ hold simultaneously.

\item For a function $f:X\to \CC$ and positive-valued
function $g:X\to (0, \infty)$, write $f = \mathcal{O}(g)$ if there exists a
constant $C>0$ such that $|f(x)| \le C g(x)$ for all $x\in X$. We will
also write $f = \mathcal{O}_{\delta}(g)$ if the implicit constant depends on
$\delta$.  For two functions $f, g:X\to \CC$ such that $g(x)\neq0$ for
all $x\in X$ we write $f = o(g)$ if $\lim_{x\to\infty}f(x)/g(x)=0$.
\end{enumerate}

\section{The circle method of Hardy and Littlewood}
The circle method was developed by Hardy and Littlewood in order to handle additive problems in the number theory. It was significantly improved by Vinogradov \cite{Vino} in 1928 and was used by him to obtain the asymptotic formula in the Goldbach ternary representation problem. The main innovation introduced by Vinogradov was the usage of the exponential sums instead of infinite series which greatly simplified the method. Usually the application of the circle method is followed by using estimates for the exponential sums like Weyl's or Vinogradov's inequality. Over the years the circle method has become a widely used tool in  analytic number theory as well as in harmonic analysis, see \cite{Woo2,Woo0,B3,miszwr1} and references given there. We recommend an excellent treatise on the subject due to Vaughan \cite{Vo}. For more general view on the analytic number theory we refer to the monograph by Iwaniec and Kowalski \cite{IK}.

Now, let us show how to derive the asymptotic formula \eqref{eq:thm:main} with the aid of the circle method. Let $k\ge2$ be a fixed integer and denote
\[
S_N=\{(n_1, n_2, \ldots, n_d)\in\N^d: \ n_1^k+n_2^k+\ldots+n_d^k=N\}.
\]
Observe, that for $\X_N=\lfloor N^{1/k}\rfloor$ one can write
\begin{align}
	\nonumber r_k(N)&=\sum_{(n_1, n_2, \ldots, n_d)\in\N^d}\ind{S_N}(n_1, n_2, \ldots, n_d)\\
	\nonumber &=\sum_{n_1=1}^{\X}\ldots\sum_{n_d=1}^{\X}
	\int_{0}^{1}\ex{ \xi(n_1^k+n_2^k+\ldots+n_d^k)}e^{-2\pi i \xi N}{\rm d}\xi\\
	\label{asymp1}  &=\int_{0}^{1}\big(f_{\X_N}(\xi) \big)^d \ex{- \xi N}{\rm d}\xi.
\end{align}
Therefore, our task is reduced to finding the asymptotics for the integral \eqref{asymp1}. We follow the approach taken by Hardy and Littlewood and decompose the unit interval $[0,1]$ into two disjoint sets, called the major arcs $\mathfrak{M}_{X_N}$ and the minor arcs $\mathfrak{m}_{X_N}$, and evaluate the integral over both sets separately. The major arcs consist of such real numbers $\xi\in[0,1]$ which can be "well approximated" by rational numbers $a/q$ with $(a,q)=1$. For $\xi\in\mathfrak{M}_{X_N}$ we are able to show that
\begin{equation*}
	f_{\X_N}(\xi)\approx G(a/q)\int_0^{N^{1/k}}\ex{ (\xi-a/q)x^k}{\rm d}x.
\end{equation*}
On the other hand, on the minor arcs, which are the complement of the major arcs, the integral \eqref{asymp1} is negligible.

Following this idea, for fixed $N\in\N$ and $\alpha\in(0,1/4)$ we
define the family of the \textit{major arcs}
\[
\mathfrak{M}_{\X_N}:=\bigcup_{1\le q\le \X_N^{\alpha}}\bigcup_{\genfrac{}{}{0pt}{}{a=1}{(a, q)=1}}^q\mathfrak{M}_{\X_N}(a/q),
\]
where
\[
\mathfrak{M}_{\X_N}(a/q):=\big\{\xi\in[0, 1]: |\xi-a/q|\le \X_N^{-k+\alpha}\big\}\quad\text{with}\quad q\leq X_N^\alpha.
\]
We see that if $a/q$ varies over the rational fractions with small
denominators ($1\le q\le \X_N^{\alpha}$ and $(a, q)=1$) then
$\mathfrak{M}_{\X_N}(a/q)$ are disjoint.

The \textit{minor arcs} will be the set
$$\mathfrak{m}_{\X_N}=[0, 1]\setminus\mathfrak{M}_{\X_N}.$$

In view of this partition we obtain that
\begin{equation}
		\label{asymp2}
	\begin{split}
 r_k(N)&=\int_{\mathfrak{M}_{\X_N}}\big(f_{\X_N}(\xi)\big)^d \ex{-\xi N }{\rm d}\xi+\int_{\mathfrak{m}_{\X_N}}\big(f_{\X_N}(\xi)\big)^d \ex{-\xi N }{\rm d}\xi\\
	&:=M_k(N)+m_k(N).
	\end{split}
\end{equation}
Now our task is to estimate the quantities $M_k(N)$ and $m_k(N)$.
\subsection{Minor arc estimate}
We start with estimating the term $m_k(N)$ related to the minor arcs. 
\begin{prop}
	\label{pro:marc}
	For each $k\ge 2$ there exists a constant $C_{k,\alpha}>0$ such that, for $d\ge 2^{k+1}$ and~$N\ge C_{k,\alpha} d^{5k^2/\alpha}$ we have   
	\begin{equation}
		\label{eq:pro:minor}
		m_k(N)=\mathfrak{S}(N)\frac{\Gamma\left(1+\frac{1}{k}\right)^d}{\Gamma\left(\frac{d}{k}\right)}N^{d/k-1} \cdot o(1),
	\end{equation}
	where the implicit constants in the symbol $o(1)$ are independent of the dimension $d.$
\end{prop}
We make use of a sharpened form of Weyl's inequality proven by Bourgain \cite{b1} (see also \cite[Lemma 3.1]{Dav}).
\begin{lem}[{\cite[Theorem 5]{b1}}]\label{weyline}
Suppose that $\xi\in[0,1]$ has a rational approximation $a/q$ satisfying
\[
(a,q)=1,\quad q\in\NN,\quad \Big|\xi-\frac{a}{q}\Big|\le \frac{1}{q^2}.
\]
Then for every $\e>0$ there is a constant $C_{\e}>0$ such that for every
	$X\in \NN$
	\begin{align}\label{Wineq}
		|f_X(\xi)|\le C_{\e}X^{1+\e}
		\bigg(\frac{1}{q}+\frac{1}{X}+\frac{q}{X^{k}}\bigg)^{\sigma(k)}.
	\end{align}
\end{lem}
\begin{proof}
	When $k\ge 3$ this is proved in  \cite[Theorem 5]{b1}. When $k=2$ we have $2^{1-k}=\sigma(k),$ hence the statement follows from the classical Weyl inequality, see e.g.\ \cite[Lemma 2.4]{Vo}.
\end{proof}

As a consequence of Lemma \ref{weyline} we obtain an estimate for $f_X(\xi)$ on minor arcs.

\begin{cor}\label{lem:fXmin}
Let $N\in\NN$ and assume that $\xi\in\mathfrak{m}_{\X_N}.$ Then for
	every $\e>0$ there is a constant $C_{\e}>0$ such that
	\begin{align}\label{eq:fXmin}
		|f_{X_N}(\xi)|\le C_{\e}\X_{N}^{1+\e-\alpha\sigma(k)}.
	\end{align}
\end{cor}
\begin{proof}
If $\xi\in\mathfrak{m}_{\X_N}$ then by Dirichlet's pigeonhole principle one can always find $\X_N^{\alpha}<q\le \X_N^{k-\alpha}$ and $1\le a\le q$ such that $(a, q)=1$ and
$$\Big|\xi-\frac{a}{q}\Big|\le\frac{1}{q\X_N^{k-\alpha}}\le\frac{1}{q^2}.$$
Hence, by applying Lemma~\ref{weyline} we obtain that, for every $\varepsilon>0,$ one has
\begin{equation*}
	\left|f_{\X_N}(\xi)\right|\lesssim_{\varepsilon} \X_N^{1+\e}\bigg(\frac{1}{q}+\frac{1}{\X_N}+\frac{q}{\X_N^{k}}\bigg)^{\sigma(k)}\lesssim\X_N^{1+\e} \big(\X_N^{-\alpha}+\X_N^{-1}+\X_N^{-\alpha}\big)^{\sigma(k)} \lesssim X_N^{1+\e-\alpha\sigma(k)}.
\end{equation*}
\end{proof}
The following estimate for the Gauss sums can be found in \cite{Dav}.
\begin{lem}[{\cite[Lemma 6.4]{Dav}}]\label{lem:egau}
	Assume that $1\le a\le q$ satisfies $(a,q)=1.$ Then for
	every $k\geq 2$ there is a constant $C_{k}>0$ such that
	\begin{align}\label{eq:egau}
		\big|G(a/q)\big|\le C_{k}q^{-1/k}.
	\end{align}
\end{lem}
Finally, we will need an improved version of the classical Hua's lemma (see \cite[Lemma 3.2]{Dav}) proved by Bourgain.
\begin{lem}[{\cite[Theorem 10]{b1}}]\label{lem:hua} For every $X\in\NN$ and every $\e>0$ there is a constant $C_\e>0$ such that for each $k\in\NN$ we have
	\begin{align*}
		\int_0^1\big|f_{\X}(\xi)\big|^{k(k+1)}{\rm d}\xi\leq C_\e X^{k^2+\e}.
	\end{align*}
\end{lem}

Having stated Lemmas \ref{weyline} and \ref{lem:hua} we now proceed with the proof of Proposition \ref{pro:marc}.
 
\begin{proof}[Proof of Proposition~\ref{pro:marc}]
Let us remind that $X_N=\floor{N^{1/k}}$. By applying Lemmas \ref{lem:fXmin} and \ref{lem:hua}, for any $\e>0$, one obtains
\begin{align*}
	\Big|\int_{\mathfrak{m}_{\X_N}}\big(f_{\X_N}(\xi)\big)^d \ex{-\xi N }{\rm d}\xi\Big|
	&\le\sup_{\xi\in\mathfrak{m}_{\X_N}}\big|f_{\X_N}(\xi) \big|^{d-k(k+1)}\int_0^1\big|f_{\X_N}(\xi)\big|^{k(k+1)}{\rm d}\xi\\
	&\lesssim_{k,\e} C_\e^d
	\X_N^{(1-\alpha\sigma(k)+\e)(d-k(k+1))}\X_N^{k^2+\e}\lesssim C_\e^d N^{\frac{d}{k}-1-\frac{\gamma}{k}},
\end{align*}
where $\gamma=(\alpha\sigma(k)-\e)(d-k(k+1))-\e$ is non-negative provided that we take $\e<\alpha\sigma(k)(d-k(k+1))/(1+d-k(k+1))$. Hence we just proved that
\begin{equation}\label{eq:EkN}
	|m_k(N)|\lesssim_{k,\e} C_\e^d N^{\frac{d}{k}-1-\frac{\gamma}{k}}.
\end{equation}
Since one has $\Gamma\left(1+\frac{1}{k}\right)^d\simeq_k^d 1$ and $\Gamma\left(\frac{d}{k}\right)\simeq_k^d d^{d/k}$ we see that
\begin{equation} \label{eq:mterm1}
	\frac{\Gamma\left(1+\frac{1}{k}\right)^d}{\Gamma\left(\frac{d}{k}\right)} \simeq_k^d d^{-d/k}.
\end{equation}
If $d\ge 2^{k+1}$ and $\e=\sigma(k)\alpha/20$ then $\gamma\ge \sigma(k)\alpha d/5$. Hence, by applying \eqref{eq:EkN} we get
\[
	|m_k(N)|\lesssim_k^d N^{-d\frac{\sigma(k)\alpha}{5k}}  N^{d/k-1}.
\]
In view of \eqref{eq:mterm1} and Proposition~\ref{prop:singest} we can write
\begin{align*}
	|m_k(N)|\Big(\mathfrak{S}(N) \frac{\Gamma\left(1+\frac{1}{k}\right)^d}{\Gamma\left(\frac{d}{k}\right)}N^{d/k-1}\Big)^{-1}\lesssim_k (C_k d^{1/k}N^{-\frac{\sigma(k)\alpha}{5k}})^d.  	
\end{align*}
Therefore, if we take $N\ge C_k d^{5k^2/\alpha}$ for an appropriately large $C_k>0$ we reach
\[
	|m_k(N)|\Big(\mathfrak{S}(N) \frac{\Gamma\left(1+\frac{1}{k}\right)^d}{\Gamma\left(\frac{d}{k}\right)}N^{d/k-1}\Big)^{-1}=o(1).
\]
This completes the proof of Proposition~\ref{pro:marc}.
\end{proof}
The next proposition asserts us that the singular series is uniformly bounded, up to multiplicative bounds, with respect to $N\in\NN$ and $d\geq 2^{k+1}$.  
\begin{prop}\label{prop:singest}
For any $k\geq 2$ and $d\geq 2^{k+1}$ there exists constants $a_k,A_k>0$ such that
\begin{equation}
    a_k^d\leq \mathfrak{S}(N)\leq A_k.
\end{equation}
\end{prop}
\begin{proof}
At first we note that for any $a/q$ with $(a,q)=1$ we have $|G(a/q)|\leq 1$. Hence, by Corollary~\ref{lem:egau} the fact that $d\geq 2^{k+1}$ we obtain the estimate
\begin{equation}\label{eq:indestimd}
    \Big|\sum_{\genfrac{}{}{0pt}{}{a=1}{(a, q)=1}}^qG(a/q)^d\ex{-Na/q}\Big|\leq\sum_{\genfrac{}{}{0pt}{}{a=1}{(a, q)=1}}^q|G(a/q)|^{2^{k+1}}\lesssim_k q^{1-(2^{k+1})/k}\lesssim q^{-3/2}.
\end{equation}
Consequently, for $d\geq 2^{k+1}$, we get
\begin{equation*}
    |\mathfrak{S}(N)|\leq\sum_{q=1}^\infty\Big|\sum_{\genfrac{}{}{0pt}{}{a=1}{(a, q)=1}}^qG(a/q)^d\ex{-Na/q}\Big|\lesssim_k\sum_{q=1}^\infty q^{-3/2}\lesssim_k 1.
\end{equation*}
Since $\mathfrak{S}(N)>0$ we obtain the upper bound for $\mathfrak{S}(N)$. To prove the lower bound let us define
\begin{equation*}
    \chi_N(p):=\sum_{h=0}^\infty \sum_{\genfrac{}{}{0pt}{}{a=1}{(a, p^h)=1}}^{p^h}G(a/p^h)^d\ex{-Na/p^h},\quad p\in\mathbb{P}.
\end{equation*}
It is known, see \cite[Lemma 5.7]{Nat}, that
\begin{equation}\label{prod}
    \mathfrak{G}(N)=\prod_{p\in\mathbb{P}}\chi_N(p).
\end{equation}
Moreover, by \cite[Lemma 5.10]{Nat} there is a natural number $\gamma_k$ which depends only on $k\in\NN$ such that
\begin{equation}\label{eq:estlower}
    \chi_N(p)\geq p^{\gamma_k(1-d)}>0,
\end{equation}
provided that $d\geq 2^{k+1}$. By following the proof of \cite[Lemma 5.7, eq. (5.14)]{Nat} and using the estimate \eqref{eq:indestimd} one can prove that there is $p_0=p_0(k)$ such that
\begin{equation*}
    1/2\leq\prod_{p\geq p_0}\chi_N(p)\leq 3/2.
\end{equation*}
Hence, by \eqref{prod} and \eqref{eq:estlower} we can estimate
\begin{equation*}
    \mathfrak{G}(N)\geq 1/2\prod_{p\leq p_0}\chi_N(p)\geq 1/2\prod_{p\leq p_0}p^{\gamma_k(1-d)}\geq1/2^d\prod_{p\leq p_0}p^{-\gamma_kd}=:a_k^d>0.
\end{equation*}
\end{proof}
\begin{rem}
The condition $d\geq 2^{k+1}$ may be weakened to $d\geq 2k+1$ when $k$ is odd and $d\geq 4k$ when $k$ is even. This threshold is related to the estimate \eqref{eq:estlower} which holds  for $d\geq 2k+1$ ($k$ odd) and $d\geq 4k$ ($k$ even). It is interesting whenever the singular series may be bounded independently of the dimension $d$. In \cite[Inequality (3.27)]{miszwr1} the authors were able to prove that for $k=2$ and $d\geq 16$ we have
\begin{equation*}
    \frac{1}{2}\leq \mathfrak{G}(N)\leq\frac{3}{2}.
\end{equation*}
Their argument heavily relies on the fact that $k=2$. 
\end{rem}
\subsection{Major arc estimate}
Now we need to estimate the contribution from the major arc term $M_k(N)$. Recall that for $1\leq q\leq X_N^\alpha$ we constructed the major arcs $\mathfrak{M}_{X_N}(a/q)$ for some fixed $\alpha\in(0,1/4)$ which will be adjusted later.
\begin{prop}\label{pro:Marc}
For each $k\in \NN$ there exists a constant $C_{k,\alpha}>0$ such that, for $d\ge 2^{k+1}$ and $N\ge C_{k,\alpha} d^{d/\delta}$, with some $0<\delta<\min\{1-4\alpha,\alpha\}$, we have
	\begin{equation}
		M_k(N)=\mathfrak{S}(N)\frac{\Gamma\left(1+\frac{1}{k}\right)^d}{\Gamma\left(\frac{d}{k}\right)}N^{d/k-1}(1 + o(1)),
	\end{equation}
where the implicit constants in the symbol $o(1)$ are independent of the dimension $d.$
\end{prop}
The proof of Proposition \ref{pro:Marc} is based on several approximations of the major arc term $M_k(N).$ At first, denote by
\begin{equation}
	\label{eq:vdef}
	v(\theta):=\int_{0}^{X_N} \ex{\theta z^k}{\rm d}z,\quad \theta\in\RR
\end{equation}
the continuous counterpart of $f_{X_N}(\theta)$ and let
\begin{equation}\label{eq:M1def}
\mathcal{A}^1(N):=\sum_{1\le q\le X_N^\alpha}\sum_{\substack{a=1\\(a,q)=1}}^q G(a/q)^d\int_{\mathfrak{M}_{X_N}(a/q)}\big(v(\xi-a/q)\big)^d \ex{-\xi N }{\rm d}\xi,\quad N\in\NN
\end{equation}
be our first approximation of $M_k(N).$ In order to approximate $f_{X_N}(\theta)$ by its continuous counterpart $v(\theta)$ we make use of the following approximation lemma due to van der Corput.
\begin{lem}[{\cite[Lemma 9.1]{Dav}}]\label{vander}
Suppose $g(x)$ is a real function which is twice differentiable for $A\leq x\leq B$ with $A,B\in\RR$. Suppose further that, in this interval, one has
\[
0\leq g'(x)\leq\frac12,\quad g''(x)\geq 0.
\]
Then
\begin{equation}
    \sum_{A\leq n\leq B}\ex{g(n)}=\int_A^B\ex{g(x)}{\rm d}x+\mathcal{O}(1).
\end{equation}
\end{lem}

\begin{lem}\label{lem:M1}
For each $k\in \NN$ there exists a constant $C_{k}>0$ such that, for $d\in\NN$ and $N\ge C_{k} d^{d/\delta}$, with $0<\delta<1-4\alpha$, we have   
	\begin{equation}
		\label{eq:pro:marc}
		M_k(N)-\mathcal{A}^1(N)=\mathfrak{S}(N)\frac{\Gamma\left(1+\frac{1}{k}\right)^d}{\Gamma\left(\frac{d}{k}\right)}N^{d/k-1} \cdot o(1),
	\end{equation}
where the implicit constants in the symbol $o(1)$ are independent of the dimension $d.$
\end{lem}
\begin{proof}
Fix a fraction $a/q$ with $1\leq q\leq X_N^\alpha$ and $a\leq q$ such that $(a,q)=1$. Consider the corresponding major arc $\mathfrak{M}_{\X_N}(a/q)$. Let us collect together those values of $n$ in the sum defining $f_{X_N}(\xi)$ which are in the same residue class and write
\begin{equation}
\label{eq:fXdecom}
f_{X_N}(\xi)=\sum_{r=1}^q\ex{a r^k/q}\sum_{-\frac{r}{q}<n\le \frac{\X_N-r}{q}}\ex{(\xi-a/q)(qn+r)^k}.
\end{equation}
For $x\in(-q^{-1}r,q^{-1}(\X_N-r)]$ we denote 
\[
g(x):=(\xi-a/q)(qx+r)^k.
\]
If $\xi-a/q\geq0$ then for $N\geq(2k)^{2k}>(2k)^{k/(1-2\alpha)}$ one has
\[
0\leq g'(x)=k(\xi-a/q)q(qx+r)^{k-1}\leq1/2
\]
since $|\xi-a/q|\leq X_N^{-k+\alpha}$ and $qx+r\leq X_N$. Moreover we have $g''(x)>0$. Hence, we may apply Lemma~\ref{vander} to the inner sum in \eqref{eq:fXdecom}. In the case when $\xi-a/q<0$ we use Lemma~\ref{vander} to the complex conjugate sum of the inner sum. Consequently, by noting that $q\le X_N^{\alpha}$ we conclude that
\begin{equation*}
    f_{X_N}(\xi)=\sum_{r=1}^q\ex{a r^k/q}\int_{-r/q}^{(\X_N-r)/q} \ex{ (\xi-a/q) (qx+r)^k}{\rm d}x+\mathcal{O}(X_N^\alpha),\quad \xi\in\mathfrak{M}_{\X_N}(a/q).
\end{equation*}
If we change variable $qx+r=y$ in the above integral we obtain 
\begin{equation*}
    f_{X_N}(\xi)=G(a/q)v(\xi-a/q)+\mathcal{O}(X_N^\alpha),\quad \xi\in\mathfrak{M}_{\X_N}(a/q).
\end{equation*}
Now, by using the identity $x^d-y^d=(x-y)(x^{d-1}+x^{d-2}y+\cdots+xy^{d-2}+y^{d-1})$ and the obvious bounds $|f_{\X_N}(\xi)|\le X_N$ and $|G(a/q)\cdot v(\xi-\frac
aq)|\le X_N$ we get
    \begin{equation}\label{eq:lem:M1:1}
    \Big|\big(f_{\X}(\xi)\big)^d -\big(G(a/q)\cdot v(\xi-a/q)\big)^d \Big|\lesssim d \cdot X_N^{\alpha+d-1},\qquad \xi\in \mathfrak{M}_{\X_N}(a/q).
\end{equation}

Using \eqref{eq:lem:M1:1} and recalling definitions \eqref{asymp2} and \eqref{eq:M1def} we estimate
\begin{align*}
	|M_k(N)-\mathcal{A}^1(N)| \lesssim d\, X_N^{\alpha+d-1} \cdot \sum_{1\le q\le \X_N^{\alpha}}\sum_{\genfrac{}{}{0pt}{}{a=1}{(a, q)=1}}^q \left| \mathfrak{M}_{X_N}(a/q) \right|\lesssim d\, X_N^{d-k+4\alpha-1}.
\end{align*}
Hence we see that
\begin{equation}\label{eq:lem:M1:2}
N^{-d/k+1}	|M_k(N)-\mathcal{A}^1(N)|\lesssim d\cdot X_N^{4\alpha-1},
\end{equation}
and using \eqref{eq:mterm1} we get
\begin{align*}
	\mathfrak{S}(N)^{-1}\frac{\Gamma\left(\frac{d}{k}\right)}{\Gamma\left(1+\frac{1}{k}\right)^d} N^{-d/k+1}	|M_k(N)-\mathcal{A}^1(N)|\lesssim_k^d d^{d/k}\cdot X_N^{4\alpha-1}.
\end{align*}
Therefore, taking $0<\delta<1-4\alpha$ we obtain for $X_N\ge d^{d/(k\delta)}$ the estimate
\[
	\mathfrak{S}(N)^{-1}\frac{\Gamma\left(\frac{d}{k}\right)}{\Gamma\left(1+\frac{1}{k}\right)^d} N^{-d/k+1}	|M_k(N)-\mathcal{A}^1(N)|\lesssim_k^d X_N^{4\alpha-1+\delta}.
\]
The above justifies \eqref{eq:pro:marc} and completes the proof of Lemma \ref{lem:M1}.
\end{proof}

\begin{rem}
	\label{rem:problem}
	The threshold $N \ge C_{k} d^{d/\delta}$ appears in Lemma \ref{lem:M1} because of the method in which we estimate the difference between the exponential sum and the exponential integral. The proof is based on the simple one-dimensional estimate \eqref{eq:lem:M1:1} which translates into the decay of order $d X_N^{4\alpha-1}$ in \eqref{eq:lem:M1:2}. The same issue appears if we consider $(f_{X_N}(\xi))^d$ as a sum over $\NN^d$ and use mean value theorem to estimate the difference between the sum and the integral. In order to obtain a polynomial threshold of the form $N> C_k d^{ck^j},$ for some  $j\in \NN$ we would need a decay in \eqref{eq:lem:M1:2} of an order $X_N^{-d\gamma},$ where $\gamma$ is a non-negative constant. This issue seems to persist whenever the  method employed to solve the Waring problem uses a one-dimensional error term bound or the mean value theorem. It is worth mentioning that this is an approach commonly used in the literature. For instance the threshold $N \ge C_k d^{d/\delta}$ also appears if one applies methods developed by Magyar \cite{Mag1}. This is because of \cite[Lemma 4]{Mag1} which  uses a mean value theorem at the beginning of p.\ 929. Similarly,  the higher order asymptotics formula in Waring's problem proved recently by Vaughan and Wooley \cite{VaWoo} makes use of a one-dimensional error term bound in \cite[eq.\ (3.3), (3.4)]{VaWoo}. This then translates to a requirement that $N\ge C_k d^{d/\delta}$ in order to have a dimension-free control on the error bound in \cite[Lemma 3.1]{VaWoo}. On the other hand in the proof of \cite[Theorem 3.1, eq. 3.17]{miszwr1} the authors do not use such a reduction to a one-dimensional estimate. They were able to resort explicitly to the Poisson summation formula in the case $k=2.$  This leads to a much lower threshold for $k=2$. Whether this is possible for general $k$ is not clear to us. 
\end{rem}

The second approximation of $M_k(N)$ is based on a replacement of the range of integration in $\mathcal{A}^1.$ We let 
\begin{equation}
	\label{eq:M2def}
	\mathcal{A}^2(N):=\sum_{1\le q\le \X_N^{\alpha}}\sum_{\genfrac{}{}{0pt}{}{a=1}{(a, q)=1}}^q G(a/q)^d\int_{-\infty}^\infty\big(v(\xi-a/q)\big)^d \ex{-\xi N }{\rm d}\xi.
\end{equation}
For proving Lemma~\ref{lem:M2} we make use of an estimate for the function $v$ defined in \eqref{eq:vdef} and justified in \cite{Woo1} (see also \cite[p. 21]{Dav}). 
\begin{lem}[{\cite[Lemma 6.1]{Woo1}}]\label{lem:vlem} 
For each $k\in\NN$ and for any $\xi\in\R$ we have
\begin{equation}\label{eq:veq}
	|v(\xi)|\lesssim_k X_N(1+X_N^{k}|\xi|)^{-1/k}.
\end{equation}
\end{lem}
Now we can state our second approximation of $M_k(N)$.
\begin{lem}\label{lem:M2}
For each $k\in \NN$ there exists a constant $C_{k,\alpha}>0$ such that, for $d\ge 2^{k+1}$ and $N\ge C_{k,\alpha} d^{d/\delta}$, with $0<\delta<\min\{1-4\alpha,\alpha\}$, we have 
\begin{equation}\label{eq:pro:marc2}
	M_k(N)-\mathcal{A}^2(N)=\mathfrak{S}(N)\frac{\Gamma\left(1+\frac{1}{k}\right)^d}{\Gamma\left(\frac{d}{k}\right)}N^{d/k-1} \cdot o(1),
\end{equation}
where the implicit constants in the symbol $o(1)$ are independent of the dimension $d.$
\end{lem}
\begin{proof}
By Lemma~\ref{lem:M1} it is enough to show that
\begin{equation}\label{eq:pro:m1m2}
	\mathcal{A}^1(N)-\mathcal{A}^2(N)=\mathfrak{S}(N)\frac{\Gamma\left(1+\frac{1}{k}\right)^d}{\Gamma\left(\frac{d}{k}\right)}N^{d/k-1} \cdot o(1),\qquad \textrm{for}\quad N\ge C_{k,\alpha} d^{k/\alpha}
\end{equation}
for a large enough universal constant $C_{k,\alpha}>0.$
Now, from Lemma~\ref{lem:vlem} and the definition of $\mathfrak{M}_{X_N}(a/q)$  it follows that
\begin{align*}
	\Big|\int_{\R\setminus \mathfrak{M}_{X_N}(a/q)}\left(v(\xi-a/q)\right)^d \ex{-\xi N }{\rm d}\xi\Big|& \lesssim^d \int_{|\xi|> X_N^{-k+\alpha}} X_N^d(1+X_N^{k}|\xi|)^{-d/k}\,{\rm d}\xi \lesssim^d \int_{|\xi|> X_N^{-k+\alpha}} |\xi|^{-d/k}\,{\rm d}\xi\\
& \lesssim^d (X_N^{-k+\alpha})^{-d/k+1}=X_N^{d-k-(d/k-1)\alpha}.
\end{align*}
Consequently, summing over the pairs $(a,q)$ and using Lemma~\ref{lem:egau} we reach
\begin{align*}
	|\mathcal{A}^1(N)-\mathcal{A}^2(N)|&\leq C_k^d  X_N^{d-k-(d/k-1)\alpha}\sum_{1\le q\le \X_N^{\alpha}}\sum_{\genfrac{}{}{0pt}{}{a=1}{(a, q)=1}}^q q^{-d/k}\leq  C_k^d X_N^{d-k-(d/k-1)\alpha} X_N^{\alpha(2-d/k)}.
\end{align*}
We use \eqref{eq:mterm1} to get
\begin{align*}
	\mathfrak{S}(N)^{-1}\frac{\Gamma\left(\frac{d}{k}\right)}{\Gamma\left(1+\frac{1}{k}\right)^d} N^{-d/k+1}	|\mathcal{A}^1(N)-\mathcal{A}^2(N)| &\leq  C_k^d d^{d/k} X_N^{-(d/k-1)\alpha+\alpha(2-d/k)}\lesssim_k^d (dX_N^{-\alpha})^{d/k}X_N^{\alpha(3-d/k)}.
\end{align*}
Now, taking $N\ge C_{k,\alpha} d^{k/\alpha}$ for a large enough $C_{k,\alpha}>0$ we obtain 
\begin{align*}
	\mathfrak{S}(N)^{-1}\frac{\Gamma\left(\frac{d}{k}\right)}{\Gamma\left(1+\frac{1}{k}\right)^d} N^{-d/k+1}	|\mathcal{A}^1(N)-\mathcal{A}^2(N)|=o(1).
\end{align*}
since $(3-d/k)<-1$ for $k\ge2$ and $d\ge 2^{k+1}$. This completes the proof of Lemma \ref{lem:M2}.
\end{proof}

Change of variable shows that the definition \eqref{eq:M2def} may be rewritten as
\begin{equation*}
	\mathcal{A}^2(N)=\sum_{1\le q\le \X_N^{\alpha}}\sum_{\genfrac{}{}{0pt}{}{a=1}{(a, q)=1}}^q G(a/q)^d\ex{-Na/q}\int_\R\left(v(\xi)\right)^d \ex{-\xi N }{\rm d}\xi.
\end{equation*}
It is well known that, for $d\ge k+1$ we have 
\[
\int_{\R}\left(v(\xi)\right)^d \ex{-\xi N }{\rm d}\xi=\frac{\Gamma\left(1+\frac{1}{k}\right)^d}{\Gamma\left(\frac{d}{k}\right)}N^{d/k-1},
\]
see e.g.\ \cite[Theorem 4.1]{Dav} or \cite[Lemma 6.3]{Woo1}. Thus, we obtain
\begin{equation}
	\label{eq:M2def'}
	\mathcal{A}^2(N)=\sum_{1\le q\le \X_N^{\alpha}}\sum_{\genfrac{}{}{0pt}{}{a=1}{(a, q)=1}}^q G(a/q)^d\ex{-Na/q}\cdot\frac{\Gamma\left(1+\frac{1}{k}\right)^d}{\Gamma\left(\frac{d}{k}\right)} N^{d/k-1}.
\end{equation}
The last step in proving Proposition~\ref{pro:Marc} is the replacement of the truncated singular series 
\[
\sum_{1\le q\le \X_N^{\alpha}}\sum_{\genfrac{}{}{0pt}{}{a=1}{(a, q)=1}}^q G(a/q)^d\ex{-Na/q}
\]
by the singular series \eqref{eq:sise}. Let $\mathcal{A}^3(N)$ by our final approximation of $M_k(N)$ given by
    \begin{equation}\label{eq:M3def}
		\mathcal{A}^3(N):=\mathfrak{S}(N)\frac{\Gamma\left(1+\frac{1}{k}\right)^d}{\Gamma\left(\frac{d}{k}\right)} N^{d/k-1}.
    \end{equation}

\begin{lem}\label{lem:M3}
For each $k\in \NN$ there exists a constant $C_{k,\alpha}>0$ such that, for $d\ge 2^{k+1}$ and $N\ge C_{k,\alpha} d^{d/\delta}$, with $0<\delta<\min\{1-4\alpha,\alpha\}$, one has
	\begin{equation}\label{eq:pro:marc3}
		M_k(N)-\mathcal{A}^3(N)=\mathfrak{S}(N)\frac{\Gamma\left(1+\frac{1}{k}\right)^d}{\Gamma\left(\frac{d}{k}\right)}N^{d/k-1} \cdot o(1),
	\end{equation}
where the implicit constants in the symbol $o(1)$ are independent of the dimension $d$.
\end{lem}

\begin{proof}
By Lemma~\ref{lem:M2} it suffices to justify that
	\begin{equation}\label{eq:pro:m2m3}
		\mathcal{A}^3(N)-\mathcal{A}^2(N)=\mathfrak{S}(N)\frac{\Gamma\left(1+\frac{1}{k}\right)^d}{\Gamma\left(\frac{d}{k}\right)}N^{d/k-1} \cdot o(1),\qquad \textrm{for}\quad N\ge C_{k,\alpha} d^{d/\delta}
	\end{equation}
	with some large universal constant $C_{k,\alpha}>0.$
Clearly, by Proposition~\ref{prop:singest} we have
    \begin{align*}
	    |\mathcal{A}^3(N)-\mathcal{A}^2(N)|\leq a_k^{-d} \Big| \sum_{q> \X_N^{\alpha}}\sum_{\genfrac{}{}{0pt}{}{a=1}{(a, q)=1}}^q  G(a/q)^d\ex{-Na/q}\Big|\cdot \mathfrak{S}(N)\frac{\Gamma\left(1+\frac{1}{k}\right)^d}{\Gamma\left(\frac{d}{k}\right)}N^{d/k-1},
    \end{align*}
and as a consequence \eqref{eq:pro:m2m3} will follow if we show that
    \begin{equation}\label{eq:sita}
	    a_k^{-d}\sum_{q> \X_N^{\alpha}}\sum_{\genfrac{}{}{0pt}{}{a=1}{(a, q)=1}}^q  G(a/q)^d\ex{-Na/q}=o(1).
    \end{equation}
By the estimate \eqref{eq:indestimd} we obtain
    \begin{align*}
	    a_k^{-d}\Big|\sum_{q> \X_N^{\alpha}}\sum_{\genfrac{}{}{0pt}{}{a=1}{(a, q)=1}}^q  G(a/q)^d\ex{-Na/q}\Big|\lesssim_k a_k^{-d} \sum_{q> X_N^{\alpha}}q^{-3/2}\leq C_k^dX_N^{-\alpha/10},
    \end{align*} 
for some $C_k>0$. Now if we take $N\ge C_{k,\alpha} d^{d/\delta}$,  for a large $C_{k,\alpha}>0$, we see that the estimate \eqref{eq:sita} holds.
\end{proof}
Now we are ready to give the proof of Proposition~\ref{pro:Marc} and Theorem~\ref{thm:main}

\begin{proof}[Proof of Proposition~\ref{pro:Marc} and Theorem~\ref{thm:main}]
We see that the Proposition \ref{pro:Marc} follows by Lemma~\ref{lem:M3}.

In order to prove Theorem~\ref{thm:main} we use Propositions~\ref{pro:marc} and \ref{pro:Marc} with $\alpha=1/5$ and $\delta=1/20$ to obtain that for $d\geq 2^{k+1}$ and $N\gtrsim_k d^{20d}$ we have 
\begin{equation*}
	r_k(N)=\mathfrak{S}(N)\frac{\Gamma\left(1+\frac{1}{k}\right)^d}{\Gamma\left(\frac{d}{k}\right)}N^{d/k-1}(1 + o(1)).
\end{equation*}
\end{proof}
\subsection{Comments and questions}
Let us state some comments and questions concerning our results.
\begin{enumerate}[label*={\arabic*}.]
\item By choosing an appropriate small $\alpha\in(0,1/4)$ we could achieve in Lemma~\ref{lem:M1} a better threshold. Namely, we would get that it is enough to take $N\geq C_k d^{(1+\delta)d}$ for some $C_k>0$ and small $\delta>0$. This would imply that for a large $k\in\NN$ and $d\geq 2^{k+1}$ it is enough to take $N\geq C_k d^{(1+\delta)d}$ in Theorem~\ref{thm:main}. It is still a growth of order $d^d\,$  instead of the expected by us polynomial growth $d^{k^j}$, for some $j\in\NN$. It seems that our approximation method in Lemma~\ref{lem:M1} is not good enough to achieve a lower order of the threshold for $N$. In the case of $k=2$ we know that the formula \eqref{eq:thm:main} holds for $d\ge 16$ and $N\gtrsim d^3$. Is that the optimal threshold for $N$? What about the other values of $k$? 
\item Denote by $G_{\rm multi}(k)$ the smallest integer such that for $d\geq G_{\rm multi}(k)$ the multiplicative formula \eqref{eq:thm:main} holds. What is the optimal value of $G_{\rm multi}(k)$? Our theorem provides the bound $G_{\rm multi}(k)\le 2^{k+1}$ with the threshold $N\gtrsim_k d^{20d}$. By a careful analysis of our proof one can note that for large $k\in\NN$ it is enough to take $G_{\rm multi}(k)\le k(k+1)$ for $N\gtrsim_k d^{d^2}$. Is that trade-off between $G_{\rm multi}(k)$ and the threshold for $N$ necessary? In the case of the usual additive formula \eqref{eq:sise} it is known, see \cite[eq. (6.16)]{b1}, that for large $k\in\NN$ it holds for $d\geq G_{\rm add}(k)$ with
\begin{equation*}
    G_{\rm add}(k)\le k^2-k+\mathcal{O}(k).
\end{equation*}
What is the relation between $G_{\rm multi}(k)$ and $G_{\rm add}(k)$? It is natural to expect that $G_{\rm add}(k)\leq G_{\rm multi}(k)$ since we need some place to get rid of the dependence on the dimension. Does the equality $G_{\rm add}(k)=G_{\rm multi}(k)$ hold?
\item  Let us consider the generalized Waring problem
\begin{equation}\label{eq:genwar}
    c_1n_1^k+c_2n_2^k+\ldots+c_dn_d^k=N,
\end{equation}
where $c_1,\ldots,c_s$ are given positive integers and $n_1,\ldots, n_d$ are arbitrary positive integers. We assume that $c_1,\ldots,c_s$ do not all have a common factor greater than $1$. Additionally, in order to ensure solvability, we assume that the congruence
\begin{equation*}
    c_1n_1^k+c_2n_2^k+\ldots+c_dn_d^k\equiv N\quad (\rm{mod}\,\, p^h)
\end{equation*}
has a solution for every $p\in\mathbb{P}$ and large $h\in\NN$. For more details see \cite[Chapter 7]{Dav} or \cite[Chapter 11 and 12]{Woo1}. Denote by $c=(c_1,c_2,\ldots, c_d)$ and $\boldsymbol{c}=\max_{1\leq i\leq d}c_i$. The singular series for the generalized Waring problem \eqref{eq:genwar} is given by
\begin{equation*}
    \mathfrak{S}_c(N):=\sum_{q=1}^\infty\sum_{\genfrac{}{}{0pt}{}{a=1}{(a, q)=1}}^q\prod_{i=1}^d G(c_ia/q)\ex{-Na/q} .
\end{equation*}
Let $r_{k}^c(N)$ denote the number of solutions of the equation \eqref{eq:genwar}. By following the presented approach and by using some facts from \cite[Chapter 11 and 12]{Woo1} one can prove the following.
\begin{thm}\label{thm:main2}
For each $k\in \NN$ with $k\ge 2$ there exists a constant $C_{\boldsymbol{c},k}>0$ such that, for $d\ge 2^{k+1}$ and $N\ge C_{\boldsymbol{c},k} d^{20d}$, we have   
\begin{equation}\label{eq:thm:main2}
    r_k^{c}(N)=\frac{\mathfrak{S}_c(N)}{(c_1 c_2\cdots c_d)^{1/k}}\frac{\Gamma\left(1+\frac{1}{k}\right)^d}{\Gamma\left(\frac{d}{k}\right)}N^{d/k-1}(1 + o(1)),
\end{equation}
where the implicit constants in the symbol $o(1)$ depend only on $\boldsymbol{c}$ and $k$. 
\end{thm}
\end{enumerate}

\end{document}